\documentclass{amsart}
\usepackage{graphicx}
\usepackage{amssymb,amscd,amsthm,amsxtra}
\usepackage{latexsym}

\newtheorem{thm}{Theorem}[section]

\newtheorem{lem}[thm]{Lemma}

\theoremstyle{definition}

\theoremstyle{remark}
\newtheorem{rem}[thm]{Remark}

\numberwithin{equation}{section}

\begin{document}

\title[Partial Regularity for Singular Solutions]{Partial regularity for singular solutions to the Monge-Amp\`{e}re equation}
\author{Connor Mooney}
\address{Department of Mathematics, Columbia University, New York, NY 10027}
\email{\tt  cmooney@math.columbia.edu}

\begin{abstract}
 We prove that solutions to the Monge-Amp\`{e}re inequality 
 $$\det D^2u \geq 1$$
 in $\mathbb{R}^n$ are strictly convex away from a singular set of Hausdorff $n-1$ dimensional measure zero. Furthermore, we show this is optimal by constructing
 solutions to $\det D^2u = 1$ with singular set of Hausdorff dimension as close as we like to $n-1$. As a consequence we obtain $W^{2,1}$ regularity
 for the Monge-Amp\`{e}re equation with bounded right hand side and unique continuation for the Monge-Amp\`{e}re equation with sufficiently regular right hand side.
\end{abstract}
\maketitle

\section{Introduction}
In this paper we investigate the Hausdorff dimension of the set where Alexandrov solutions (see Section \ref{Prelims} for the precise definition) to
$$\det D^2u \geq 1$$
are not strictly convex. Recall that we say that a convex function $u$ is strictly convex at $x_0$ if there exists $L_{x_0}$, a supporting tangent plane at $x_0$, such that
$$\{u = L_{x_0}\} = x_0.$$
Our main theorem is:
\begin{thm}\label{Main}
 Assume $u$ is an Alexandrov solution to
 $$\det D^2u \geq 1$$ 
 in $B_1 \subset \mathbb{R}^n$. Then $u$ is strictly convex away from a singular set $\Sigma$ with
 $$\mathcal{H}^{n-1}(\Sigma) = 0.$$
\end{thm}
We show this is optimal by constructing solutions to $\det D^2u = 1$ with singular set of Hausdorff dimension as close as we like to $n-1$.
This result is interesting especially for $n \geq 3$ since it is well-known that in two dimensions solutions to $\det D^2u \geq 1$ are strictly convex.

Previous results on the singularities of convex functions include those of Alberti, Ambrosio and Cannarsa (see \cite{A}, \cite{AAC}), who show that the
nondifferentiability set of a semi-convex function is $n-1$ rectifiable. Theorem \ref{Main} may be viewed as a strengthening
of this result when we have positive lower {\it and} upper bounds on $\det D^2u$, in which case Caffarelli's regularity theory (see below) gives differentiability 
at points of strict convexity. (In fact, if $\det D^2 u = 1$ in the Alexandrov sense then $\Sigma$ is precisely the set where $u$ fails to be a classical solution.)
However, it is important to note that points in $\Sigma$ may still be
points of differentiability for $u$ (see for example the Pogorelov solution to $\det D^2u = 1$ below), and 
without an upper bound on $\det D^2u$ the points of non-differentiability for $u$ may not be in $\Sigma$ (take
for example $u = |x|^2 + |x_n|$, which solves $\det D^2u \geq 1$ and is strictly convex everywhere).

Theorem \ref{Main} has several applications to the regularity theory for singular solutions to the Monge-Amp\`{e}re equation with bounded right hand side, which we now describe.

Caffarelli developed a regularity theory of solutions to
$$\det D^2u = f \quad \text {in } \Omega, \quad \quad \lambda \leq f \leq \Lambda$$
at points where $u$ is strictly convex. We briefly summarize the main results. We define a section of $u$ at $x$ with height $h$ and slope $p$ by
$$S_{h,p}^u(x) = \{y \in \Omega : u(y) < u(x) + p \cdot (y-x) + h\}$$
for some subgradient $p$ at $x$.
If $u$ is strictly convex at $x$ then we can find a subgradient $p$ such that the supporting plane of this slope touches only at $x$, and then take
$h$ small enough that $S_{h,p}^u(x) \subset\subset \Omega$. In this setting, Caffarelli (\cite{C1},\cite{C2}) showed that
\begin{enumerate}
 \item $u$ is strictly convex in $S_{h,p}^u(x)$ and $u \in C^{1,\alpha}_{loc}(S_{h,p}^u(x))$,
 \item If $f \in C^{\alpha}(\Omega)$ then $u \in C^{2,\alpha}_{loc}(S_{h,p}^u(x))$, and
 \item For every $q > 1$ there is some $\epsilon(q) > 0$ such that if $|f-1| < \epsilon$ then $u \in W^{2,q}_{loc}(S_{h,p}^u(x))$.
\end{enumerate}

However, these regularity theorems fail at points where $u$ is not strictly convex. Consider the well-known Pogorelov examples on 
$B_{1} \subset \mathbb{R}^n, \quad n \geq 3$ which degenerate along $x' = (x_1,...,x_{n-1}) = 0$. One constructs these examples by seeking solutions of the form
$|x'| + |x'|^{\beta}g(x_n)$ and $|x'|^{\alpha}f(x_n)$. The first is
$$|x'| + |x'|^{n/2}(1+x_n^2),$$
which solves $\lambda \leq \det D^2u \leq \Lambda$ but is merely Lipschitz. The second is
$$|x'|^{2-2/n}(1+x_n^2),$$
which solves $\det D^2u = f$ with $f$ strictly positive and smooth, but is only $C^{1,\alpha}$ for $\alpha = 1-2/n$ and $W^{2,p}$ for $p < \frac{n(n-1)}{2}$.

In \cite{C3}, Caffarelli generalizes these examples to solutions that degenerate along subspaces of any dimension less than $\frac{n}{2}$, and shows
that it is not possible to find solutions degenerating on subspaces of dimension $\frac{n}{2}$ or higher. We provide a short proof in the next section (see Lemma \ref{MaxDimension}).
If $u$ agrees with a linear function $L$ on a $k$-dimensional set, we say that $\{u=L\}$ is a $k$-dimensional singularity.
Our proof of Theorem \ref{Main} in fact shows that the collection of $k$-dimensional singularities has Hausdorff $n-k$ dimensional measure zero (see Remark \ref{HigherDimSing}).

Since we cannot hope for $C^{1}$ regularity or $W^{2,p}$ regularity of singular solutions to $\lambda \leq \det D^2u \leq \Lambda$ for large $p$, it is natural
to ask what we can show about the integrability of the second derivatives. De Philippis, Figalli and Savin (\cite{DFS},\cite{DF}) recently 
showed $W^{2,1+\epsilon}$ regularity of strictly convex solutions to $\lambda \leq \det D^2u \leq \Lambda$, where $\epsilon$ depends only on $\lambda,\Lambda$
and $n$. Our main theorem rules out the possibility that the second derivatives concentrate on $\Sigma$:

\begin{thm}\label{W21}
 Let $u$ be an Alexandrov solution to
 $$\lambda \leq \det D^2u \leq \Lambda$$
 in $B_1 \subset \mathbb{R}^n$. Then $u \in W^{2,1}_{loc}(B_1)$.
\end{thm}
We also show that Theorem \ref{W21} is optimal by proving that the examples giving optimality of Theorem \ref{Main} are not in $W^{2,1+\epsilon}$ for $\epsilon$ as small as we like.

A second consequence of Theorem \ref{Main} is that the points of strict convexity for $u$ form a connected set when
$f$ is bounded away from $0$ (see Lemma \ref{Connected}). If $f$ is sufficiently regular we obtain unique continuation for the Monge-Amp\`{e}re equation:

\begin{thm}\label{UniqueContinuation}
 Assume that $u$ and $v$ are Alexandrov solutions to
 $$\det D^2u = \det D^2v = f$$
 in an open connected set $\Omega \subset \mathbb{R}^n$, with $f \in C^{1,\alpha}(\Omega)$ strictly positive.
 If $u = v$ on an open subset of $\Omega$, then $u \equiv v$ in $\Omega$.
\end{thm}

To our knowledge, these are the first Sobolev regularity and unique continuation results for singular solutions to the Monge-Amp\`{e}re equation.

The paper is organized as follows. In Section \ref{Prelims} we present basic geometric properties of the sections of solutions to $\det D^2u \geq 1$. In particular, we present
an important estimate on the volume growth of sections that are not compactly contained and relate the volume of compactly contained sections
to the Monge-Amp\`{e}re mass of these sections. 
In Section \ref{PfofMain} we use these results at singular points together with the useful technique of replacing $u$ by $u + \frac{1}{2}|x|^2$ 
to prove Theorem \ref{Main}. In Section \ref{Examples} we construct, for any $\delta$, a solution to $\det D^2u = 1$
with a singular set of Hausdorff dimension $n-1-\delta$, which shows that our main theorem is optimal. In Section \ref{W21Section} we use Theorem \ref{Main} to
prove Theorem \ref{W21} and we show that the examples constructed in Section \ref{Examples} are not in $W^{2,1+\epsilon}$ for $\epsilon$ as small as we like,
which shows that $W^{2,1}$ regularity is optimal. Finally, in Section \ref{UCsection} we prove Theorem \ref{UniqueContinuation} by applying
a classical unique continuation theorem in the set of strict convexity.

In future work we intend to present a more precise, quantitative version of our main theorem to obtain $L\log L$ estimates
for the second derivatives of singular solutions to $\lambda \leq \det D^2u \leq \Lambda$.

\section{Preliminaries}\label{Prelims}

We first recall the precise definition of Alexandrov solutions. Any convex function $v : \Omega \subset \mathbb{R}^n \rightarrow \mathbb{R}$ has an associated Borel measure $Mv$, called the 
Monge-Amp\`{e}re measure, defined by
$$Mv(A) = |\partial v(A)|$$
where $|\partial v(A)|$ represents the Lebesgue measure of the image of the subgradients of $v$ in $A$ (see \cite{Gut}). (We say $p \in \mathbb{R}^n$ is 
a subgradient of $v$ at $x$ if it is the slope of some supporting hyperplane to the graph of $v$ at $x$). If $v \in C^2,$ then
$$|\partial v(A)| = \int_{A} \det D^2v \, dx.$$
Given a Borel measure $\mu$, we say that $v$ is an Alexandrov solution to
$$\det D^2v = \mu$$
if $Mv = \mu$.

For a convex function $v$ defined on $\Omega \subset \mathbb{R}^n$, we define a section $S_{h,p}^v(x)$ by
$$S_{h,p}^v(x) = \{y \in \Omega: v(y) < v(x) + p \cdot(y-x) + h\}$$
for some subgradient $p$ at $x$. We now present some results on the geometry of the sections.

\begin{lem}\label{JohnsLemma}
 (John's Lemma). If $K \subset \mathbb{R}^n$ is a bounded convex set with nonempty interior, and $0$ is the center of mass of $K$, then there exists
 an ellipsoid $E$ and a dimensional constant $C(n)$ such that
 $$E \subset K \subset C(n)E.$$ 
\end{lem}

We call $E$ the John ellipsoid of $K$. There is some linear transformation $A$ such that $A(B_1) = E$, and we say that $A$ normalizes
$K$.

The next lemma is an important observation about the volume growth of sections which may not be compactly contained in $\Omega$:

\begin{lem}\label{SectionGrowth}
 Assume that $\det D^2u \geq 1$ in a bounded domain $\Omega \subset \mathbb{R}^n$. Then if $S_{h,p}^u(x)$ is any section of $u$, we have
 $$|S_{h,p}^u(x)| \leq Ch^{n/2}$$
 for some constant $C$ depending only on $n$.
\end{lem}
\begin{proof}
 Assume by translation that $0$ is the center of mass of $S_{h,p}^u(x)$. By subtracting a linear function we can assume that 
 $$p = 0, \quad u|_{\partial S_{h,0}^u(x)} \leq 0, \text { and } \quad |\min_{S_{h,0}^u(x)}u| = h.$$ 
 By John's Lemma, there is a linear transformation $A$ that normalizes $S_{h,0}^u(x)$. Let
 $$\tilde{u}(x) = |\det A|^{-2/n}u(Ax).$$
 It is easy to check that 
 $$\det D^2 \tilde{u} \geq 1, \quad \tilde{u}|_{\partial \tilde{\Omega}} \leq 0$$
 where $B_1 \subset \tilde{\Omega} \subset B_{C(n)}$. Then $\frac{1}{2}(|x|^2 - 1)$ is an upper barrier for $\tilde{u}$, so
 $$|\min_{\tilde{\Omega}}\tilde{u}| \geq \frac{1}{2}.$$
 Since $|\det A| \geq c(n)|S_{h,0}^u(x)|$, the conclusion follows.
\end{proof}

Caffarelli proved the next proposition in \cite{C3}. We provide a short proof using a technique related to our proof of the main theorem.
\begin{lem}\label{MaxDimension}
 Assume 
 $$\det D^2u \geq 1$$
 in $B_1 \subset \mathbb{R}^n$. Then $u$ cannot vanish on a subspace of dimension $\frac{n}{2}$ or higher.
\end{lem}
\begin{proof}
 Suppose $u$ vanishes on
 $$\{x_{k+1} =...=x_n=0\} \cap B_1.$$ 
 By subtracting a linear function of the form $a_{k+1}x_{k+1} + ... + a_nx_n$ we may assume that
 $u(te_n) = o(t)$. Then $S_{h,0}^u(0)$ has length $R(h)h$ in the $e_n$ direction, where $R(h) \rightarrow \infty$ as $h \rightarrow 0$. Furthermore,
 $S_{h,0}^u(0)$ has length exceeding $\frac{1}{C}h$ in the $e_{n-k},...,e_{n-1}$ directions, where $C$ is the Lipschitz constant of $u$ in $B_{1/2}$.
 Finally, $S_{h,0}^u(0)$ contains the unit ball in the subspace spanned by $\{e_1,...,e_k\}$. We conclude that
 $$|S_{h,0}^u(0)| \geq C^{-k}R(h)h^{n-k},$$
 which contradicts Lemma \ref{SectionGrowth} as $h \rightarrow 0$ for $k \geq \frac{n}{2}$.    
\end{proof}

\begin{rem}
 Lemma \ref{MaxDimension} implies in particular that every solution to $\det D^2u \geq 1$ in two dimensions is strictly convex. Furthermore, it follows
 easily that any solution to $\det D^2u \geq 1$ on some domain in $\mathbb{R}^n$ cannot agree with a linear function $l$ on any set of affine dimension $k \geq \frac{n}{2}$.
 Indeed, if not we could subtract $l$, find some point in the ($k$-dimensional) interior of $\{u=0\}$, translate to $0$ and rescale to get into the 
 setting of Lemma \ref{MaxDimension}.
\end{rem}

We conclude the section with the following variant of Alexandrov's maximum principle. In the following $c(n),C(n)$ denote small and large constants
depending only on $n$, and their values may change from line to line.

\begin{lem}\label{Alexandrov}
 Let $v$ be any convex function on a bounded domain $\Omega \subset \mathbb{R}^n$ with $v|_{\partial \Omega} = 0$. Then
 $$Mv(\Omega)\,|\Omega| \geq c(n)|\min_{\Omega}v|^n.$$
\end{lem}
\begin{proof}
 By translation assume that the center of mass of $\Omega$ is $0$. Let $A$ normalize $\Omega$ and let
 $$\tilde{v}(x) = (\det A)^{-2/n}v(Ax).$$
 Then
 $$M\tilde{v}(\tilde{\Omega}) = (\det A)^{-1}Mv(\Omega)$$
 with $B_1 \subset \tilde{\Omega} \subset B_{C(n)}$.

 The maximum of $|\tilde{v}|$ is achieved at some point $\tilde{x} \in \tilde{\Omega}$. Let $K$ be the function whose graph is the cone generated by
 $(\tilde{x},\tilde{v}(x))$ and $\partial B_{C(n)}$. By convexity,
 $$M\tilde{v}(\tilde{\Omega}) \geq |\partial K(\tilde{x})|.$$
 Since $\partial K(\tilde{x})$ contains a ball of radius at least $c(n)|\min_{\tilde{\Omega}}\tilde{v}|$, we have
 $$|\partial K(\tilde{x})| \geq c(n)|\min_{\tilde{\Omega}}\tilde{v}|^n \geq c(n)|\det A|^{-2} |\min_{\Omega}v|^n.$$
 Finally, $|\det A| \leq C(n) |\Omega|$ so the conclusion follows.
\end{proof}

\section{Proof of Theorem \ref{Main}}\label{PfofMain}
In this section assume that
$$\det D^2u \geq 1$$ 
in $B_1 \subset \mathbb{R}^n$. Fix $x \in \Sigma$ and a subgradient $p$ at $x$. By translation and subtracting a linear function assume that $x = p = 0$.
Then $\{u=0\}$ contains a line segment of some length $l$.
By Lemma \ref{SectionGrowth},
$$|S_{h,0}^u(0)| \leq C(n)h^{n/2}$$
for all $h > 0$.

Letting $v = u + \frac{1}{2}|x|^2$, it follows that 
$$|S_{h,0}^v(0)| \leq \frac{C(n)}{l}h^{\frac{n+1}{2}}$$
for all $h$ small. In fact, for any $x_0 \in \Sigma$ and subgradient $p_0$ to $v$ at $x_0$ we have
$$|S_{h,p_0}^v(x_0)| < Ch^{\frac{n+1}{2}}$$
for some $C$ which may depend on $x_0$ and $p_0$. Indeed, $p_0$ can be written as $p + x_0$ for some subgradient $p$ of $u$ at $x_0$, and one 
easily checks that
$$S_{h,p_0}^v(x_0) = S_{h,p}^{u+\frac{1}{2}|x-x_0|^2}(x_0),$$
so by subtracting a linear function with slope $p$ and translating we are in the situation described above.

Theorem \ref{Main} thus follows from the following more general result:

\begin{thm}\label{General}
 Let $v$ be any convex function on $B_1 \subset \mathbb{R}^n$ with sections $S_{h,p}^v$, and let $\Sigma_v$ denote the set of points $x$ such that
 for all subgradients $p$ at $x$, there is some $C_{x,p}$ such that
 $$|S_{h,p}^v(x)| < C_{x,p}h^{\frac{n+1}{2}}$$
 for all $h$ small. Then
 $$\mathcal{H}^{n-1}(\Sigma_v) =  0.$$
\end{thm}

\vspace{3mm}
\begin{proof}[\textbf{Proof of Theorem \ref{Main}:}]
 Let $v = u + \frac{1}{2}|x|^2$. By the discussion preceding the statement of Theorem \ref{General}, $\Sigma \subset \Sigma_v$. The conclusion 
 follows from Theorem \ref{General}.
\end{proof}
\vspace{3mm}

We briefly discuss the main ideas of the proof.
Fix $x \in \Sigma_v$ and a subgradient $p$ at $x$. In the following analysis $c,\,C$ will denote small and large constants depending on $n$ and $C_{x,p}$.
If $S_{h,p}^v(x) \subset \subset B_1$ then the definition of $\Sigma_v$ and Lemma \ref{Alexandrov} give
\begin{equation}\label{MassConcentration}
 Mv(S_h^v(x)) \geq ch^{\frac{n-1}{2}} = c(h^{1/2})^{n-1}
\end{equation}
for all $h$ small.

An important technique of the proof is to replace $v$ by $v + \frac{1}{2}|x|^2$. Since adding a quadratic can only decrease section volume, we have 
$$\Sigma_v \subset \Sigma_{v+\frac{1}{2}|x|^2}$$ and it suffices to
prove Theorem \ref{General} for this case. Then all of the sections are compactly contained in $B_1$ for $h$ small, and the diameter of
sections is at most $h^{1/2}$.  By replacing the sections $S_{h,p}^v(x)$ by $B_{\sqrt{h}}(x)$ and using a covering argument, we easily obtain that
$\Sigma_v$ has Hausdorff dimension at most $n-1$.

Lemmas \ref{NoQuadraticGrowth} and \ref{MongeAmpereMass} improve this result as follows. We aim to rule out behavior like
$$|x|^2 + |x_n|,$$
which has a singular hyperplane. For this example, the sections at $\{x_n = 0\}$ have the correct growth when we take supporting slopes
with no $x_n$-component, but the sections are too large when we take supporting slopes with $x_n$-component $1$. 

In the first lemma we use that the sections are small for {\it all} supporting planes at $x \in \Sigma_v$ to show that $v$ must grow much faster than quadratically 
in at least two directions, unlike the example above:

\begin{lem}\label{NoQuadraticGrowth}
 Assume that $v = v_0 + \frac{1}{2}|x|^2$ for some convex function $v_0$. Fix $x \in \Sigma_v$. For a supporting slope $p$ of $v$ at $x$, let
 $$d_1(h) \geq d_2(h) \geq ... \geq d_n(h)$$
 denote the axis lengths of the John ellipsoid of the section $S_{h,p}^v(x)$. Then
 $$\frac{d_{n-1}(h)}{h^{1/2}} \rightarrow 0 \text{ as } h \rightarrow 0.$$
\end{lem}

In the second lemma we use the above observation about the Monge-Amp\`{e}re mass of $v$ (inequality \ref{MassConcentration})
in the directions where $v$ grows much faster than quadratically from $x$. Since we replaced $v$ by $v+ \frac{1}{2}|x|^2$ we also know that $v$ grows
at least quadratically in the remaining directions. This allows us to cover $\Sigma_v$ with balls in which the Monge-Amp\`{e}re mass of $v$ is much larger 
than the radius to the $n-1$, giving the desired improvement.

\begin{lem}\label{MongeAmpereMass}
 Assume that $v = v_0 + \frac{1}{2}|x|^2$ for some convex function $v_0$. Fix $x \in \Sigma_v$. For any $\epsilon > 0$, there is a sequence $r_k \rightarrow 0$ such that
 $$Mv(B_{r_k}(x)) > \frac{1}{\epsilon}r_k^{n-1}.$$
\end{lem}

The proof of Theorem \ref{General} follows easily from Lemmas \ref{NoQuadraticGrowth} and \ref{MongeAmpereMass}.

\begin{proof}[\textbf{Proof of Theorem \ref{General}:}]
 Since $\Sigma_v \subset \Sigma_{v+\frac{1}{2}|x|^2}$, we may assume without loss of generality that
 $v$ has the form $v_0 + \frac{1}{2}|x|^2$ with $v_0$ convex.

 Fix $\epsilon$ small. By Lemma \ref{MongeAmpereMass}, for each $x \in \Sigma_v$ we can choose an arbitrarily small $r$ such that
 $$Mv(B_r(x)) > \frac{1}{\epsilon}r^{n-1}.$$
 Cover $\Sigma_v \cap B_{1/2}$ with such balls, and choose a Vitali subcover $\{B_{r_i}(x_i)\}_{i=1}^{N}$, i.e. a disjoint subcollection
 such that $B_{3r_i}(x_i)$ cover $\Sigma_v \cap B_{1/2}$. Then
 \begin{align*}
  \sum_{i=1}^{N} (3r_i)^{n-1} &\leq C\epsilon \sum_{i=1}^{N} Mv(B_{r_i}(x_i)) \\ 
  &\leq C\epsilon,
 \end{align*}
since $v$ is locally Lipschitz and the $B_{r_i}$ are disjoint. This means exactly that
$$\mathcal{H}^{n-1}(\Sigma_v \cap B_{1/2}) = 0.$$
The above reasoning also gives $\mathcal{H}^{n-1}(\Sigma_v \cap B_{1-\beta}) = 0$ for any $\beta$ small, but not necessarily for $\beta = 0$ 
since we only know $v$ is locally Lipschitz.
To get $$\mathcal{H}^{n-1}(\Sigma_v \cap B_1) = 0,$$ use that $\Sigma_v \cap B_1 = \cup_{k=1}^{\infty} \{\Sigma_v \cap B_{1-1/k}\}$ and 
apply countable subadditivity.
\end{proof}


We now prove Lemmas \ref{NoQuadraticGrowth} and \ref{MongeAmpereMass}.

\begin{proof}[\textbf{Proof of Lemma \ref{NoQuadraticGrowth}:}]
By translating and subtracting a linear function assume that $x = p = 0$. Assume by way of contradiction that we can find
$h_k \rightarrow 0$ and some $\delta > 0$ such that
\begin{equation}\label{ContraHypothesis}
 d_{n-1}(h_k) > \delta h_k^{1/2}
\end{equation}
for all $k$. We first show that $v$ is trapped by two tangent planes at $0$. 

Let $x_{1,k}$ and $x_{2,k}$ be the points on $\partial S_{h_k,0}^v(0)$ where the hyperplanes perpendicular to the shortest axis
of the John ellipsoid become tangent to $\partial S_{h_k,0}^v(0)$, and let $p_{1,k}$ and $p_{2,k}$ denote subgradients at these points. Since 
$$d_1(h_k)d_2(h_k)...d_n(h_k) < Ch_k^{\frac{n+1}{2}},$$
we have by the inequality \ref{ContraHypothesis} that $d_n(h_k) < \frac{C}{\delta^{n-1}}h_k$ for all $k$. By this observation and convexity we can rotate and pass to a subsequence such that 
$$p_{1,k} \rightarrow c_1(\delta)e_n, \quad p_{2,k} \rightarrow -c_2(\delta)e_n.$$
Then $v$ is trapped by the planes $\pm c(\delta)x_n$. We conclude that 
$$S_{h_k,0}^v(0) \subset \{|x_n| < C(\delta)h_k\}.$$

To complete the proof, we show that the volumes of sections obtained with tilted supporting planes are too large. Take the largest $a$ such that $v \geq ax_n$ and consider the sections
$$S_k = S_{(1+aC(\delta))h_k,ae_n}^v(0).$$
Then $S_k$ engulf $S_{h_k,0}^v(0)$. Furthermore, 
$$\sup\{|x_n|: x \in S_k\} = R_k h_k,$$
where $R_k \rightarrow \infty$ as $k \rightarrow \infty$. Indeed, if not, then for some small $\epsilon$ and a sequence $b_i \rightarrow 0$ we would have
$v(x',b_i) > (a+\epsilon)b_i$ for all $x'$. Convexity and $v(0)=0$ imply that $v > (a+\epsilon)x_n$ for all $x_n > b_i$, which in turn implies that 
$$v > (a+\epsilon)x_n,$$
contradicting the definition of $a$. 

Finally, let $x_k = (x'_k,R_kh_k) \in S_k$ be the point in $S_k$ furthest in the $e_n$ direction.
Since $v$ grows at least quadratically away from every tangent plane, we have
\begin{equation}\label{BoundedSections}
 |x'_k| < C(\delta,a)h_k^{1/2}.
\end{equation}
Explicitly, since $ae_n$ is a subgradient at $0$ and $v$ is of the form $v_0 + \frac{1}{2}|x|^2$ with $v_0$ convex, we have that $ae_n$ is a subgradient of $v_0$
at $0$, giving
$$ax_n + \frac{1}{2}|x|^2 \leq v \leq C(\delta,a)h_k + ax_n$$
in $S_k$, giving the desired bound on $|x_k'|$.

Recall that 
$$S_{h_k,0}^v(0) \subset \{|x_n| < C(\delta)h_k\} \cap B_{h_k^{1/2}}(0).$$
Take any two points $y,z$ in $\{|x_n| < C(\delta)h_k\} \cap B_{h_k^{1/2}}(0)$ a distance $\delta h_k^{1/2}$ apart,
take the lines from these points to $(x_k',R_kh_k)$ and denote the intersections of these lines with $\{x_n = C(\delta)h_k\}$ by $\tilde{y}$ and $\tilde{z}$.
Since $|y_n-z_n| < Ch_k$ and $|y-z| > \delta h_k^{1/2}$, it is obvious that $|y'-z'| > \frac{\delta}{2} h_k^{1/2}$ for $k$ large. By similar triangles
and inequality \ref{BoundedSections}, we also have 
$$|y'-\tilde{y}'| = \frac{C}{R_k}|y'-x_k'| \leq \frac{C}{R_k}h_k^{1/2},$$
and we have the same bound on $|z'-\tilde{z}'|$ (see Figure \ref{Cone}). We conclude that
\begin{equation}\label{ConeConstruction}
 |\tilde{y}-\tilde{z}| \geq |y'-z'| - |y'-\tilde{y}'| - |z'-\tilde{z}'| \geq (\delta/2 - C/R_k)h_k^{1/2}.
\end{equation}

Since $d_i(h_k) > \delta h_k^{1/2}$ for all $i \leq n-1$, inequality \ref{ConeConstruction} (applied to the center of the John ellipsoid for $S_{h_k,0}^v(0)$ and the $n-1$
dimensional ball of radius $\delta h_k$ it contains) implies that $S_k$ contains the cone with vertex 
$(x'_k,R_kh_k)$ and base containing a ball of radius $(\delta/2-C(a,\delta)/R_k)h_k^{1/2}$ on the hyperplane $\{x_n = C(\delta)h_k\}$.

\begin{figure}
 \centering
    \includegraphics[scale=0.3]{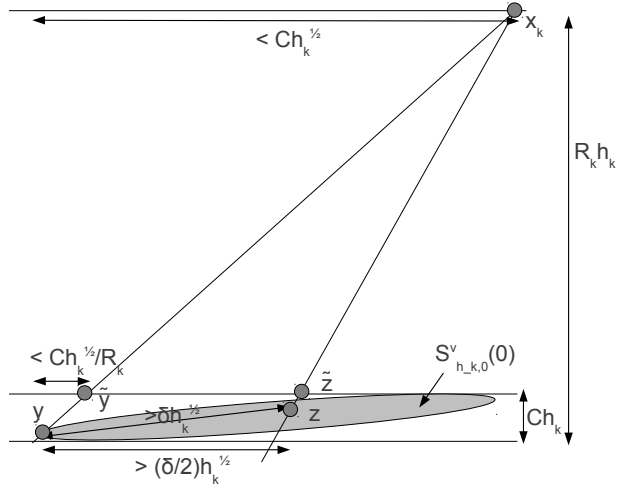}
 \caption{The cone above $\{x_n = Ch_k\}$ generated by $x_k$ and the John ellipsoid of $S_{h_k,0}^v(0)$ has a base containing a ball of radius at least $(\delta/2-C(a,\delta)/R_k)h_k^{1/2}$.}
 \label{Cone}
\end{figure}

We conclude that
$$|S_k| \geq c(\delta,a)R_kh_k^{\frac{n+1}{2}},$$
contradicting our definition of $\Sigma_v$ for $k$ large.
\end{proof}


\begin{proof}[\textbf{Proof of Lemma \ref{MongeAmpereMass}:}]
 Fix a subgradient $p$ at $x$ and let $d_1(h),...,d_n(h)$ be defined as in the statement of Lemma \ref{NoQuadraticGrowth}. Let
 $$I = \min\left\{i : \frac{d_i(h)}{h^{1/2}} \rightarrow 0 \text{ as } h \rightarrow 0\right\}.$$
 Fix $\delta$ small. Then we can find a sequence $h_k \rightarrow 0$ and $\eta$ depending only on $p$ such that
 \begin{equation}\label{FasterThanQuadratic}
  d_{I}(h_k) < \delta h_k^{1/2},
 \end{equation}
 and 
 \begin{equation}\label{AtMostQuadratic}
  d_i(h_k) > \eta h_k^{1/2}
 \end{equation}
 for all $i < I$. Rotate the axes so that the $e_i$ are the axes for the John ellipsoid of $S_{h_k,p}^v(x)$ and 
 assume by translation that $x = 0$.

 Take the restriction of $v$ to the subspace spanned by $e_{I},...,e_n$, and
 call this restriction $w$. Let 
 $$S_k^w = S_{h_k,p}^v(x) \cap \{x_1=...=x_{I-1}=0\},$$
 the slice of the section $S_{h_k,p}^v(x)$ in this subspace. Then since
 $$d_1(h_k)d_2(h_k)...d_n(h_k) \leq Ch_k^{\frac{n+1}{2}}$$
 and $v$ grows at most quadratically in the first $I-1$ directions (inequality \ref{AtMostQuadratic}), we have
 $$|S_k^w|_{\mathcal{H}^{n-I+1}} \leq \frac{C}{\eta^{I-1}}h_k^{\frac{n+2-I}{2}}.$$
 Using this and Lemma \ref{Alexandrov},
 \begin{equation}\label{SliceMA}
  Mw(S_k^w) \geq c\eta^{I-1}h_k^{\frac{n-I}{2}}.
 \end{equation}
 Finally, let $r_k = C(n)d_{I}(h_k)$, with $C(n)$ taken large enough that
 $$S_k^w \subset B_{r_k/2}(x).$$
 By strict quadratic growth in all directions, $\partial v(B_{r_k}(x))$ contains a ball of radius $r_k/2$ around every point in $\partial v(S_k^w)$. It follows that
 \begin{align*}
  Mv(B_{r_k}(x)) &\geq c(n)Mw(S_k^w)r_k^{I-1} \\
  &\geq ch_k^{\frac{n-I}{2}}r_k^{I-1} \text{ (inequality \ref{SliceMA})} \\
  &\geq \frac{c}{\delta^{n-I}}r_k^{n-1} \text{ (inequality \ref{FasterThanQuadratic})}.
 \end{align*}
 By Lemma \ref{NoQuadraticGrowth} we have $I \leq n-1$, so the conclusion follows.
\end{proof}


\begin{rem}\label{HigherDimSing}
 Replacing $\Sigma_v$ with the set $\Sigma_v^k$ of points such that
 $$|S_{h,p}^v| < C_{x,p}h^{\frac{n+k}{2}}$$
 for all $h$ small ($1 \leq k \leq n-1$) and replacing $1$ with $k$ in the preceding, one obtains that $\mathcal{H}^{n-k}(\Sigma_v^k) = 0$. 
 If $\det D^2u \geq 1$, such growth happens for $v = u + \frac{1}{2}|x|^2$ at
 points where $u$ agrees with a linear function on a $k$-dimensional subspace. This shows that the Hausdorff dimension of the
 $k$-dimensional singularities is at most $n-k$. In particular, we recover Lemma \ref{MaxDimension} since for $k \geq \frac{n}{2}$ we would
 have a $k$-dimensional singularity with Hausdorff $k$-dimensional measure $0$.
\end{rem}

\section{Examples}\label{Examples}
In this section we construct examples of solutions to $\det D^2u = 1$ in $\mathbb{R}^3$ such that $\Sigma$ has Hausdorff dimension
as close to $2$ as we like. A small modification produces the analagous examples in $\mathbb{R}^n$.

For this section, fix $\delta > 0$ small. We construct our examples in several steps, which we briefly describe:
\begin{enumerate}
 \item First, we construct functions $w$ with
 $$\det D^2w \geq 1$$
 in $\mathbb{R}^3$ that degenerate along $\{x_1=x_2=0\}$ and behave like $x_1^{2-\delta}$ along the $x_1$ axis.
 \item Next, we construct a standard $S \subset [-1,1]$ with Hausdorff dimension close to $1$ and a convex function $v$ on $[-1,1]$ 
 such that for any $x \in S$, there is a tangent line such that $v$ separates from this line faster than $r^{2-\delta}$.
 \item Finally, we get our example by solving the Dirichlet problem
 $$\det D^2 u = 1 \quad \text{ in } \Omega = \{|x'| < 1\} \times [-1,1], \quad \quad u|_{\partial \Omega} = C(\delta)(v(x_1) + |x_2|)$$
 and comparing with $w$ at points in $S \times \{0\} \times \{\pm 1\}$.
\end{enumerate}

In the following analysis $c$ and $C$ will denote small and large constants depending on $\delta$.

\textbf{Construction of $w$:} We look for a convex function $w(x_1,x_2,x_3)$ with the homogeneity
$$w(x_1,x_2,x_3) = \frac{1}{\lambda}h(\lambda^{1/\alpha}x_1,\lambda^{1/\beta}x_2)(1+x_3^2),$$
where $\alpha$ and $\beta$ satisfy $1 < \alpha,\beta < 2$ and
$$\frac{1}{\alpha} + \frac{1}{\beta} = \frac{3}{2}.$$
(It is easy to check that $\geq 3/2$ is necessary for such a function to have $\det D^2w$ bounded below). Note that this rescaling preserves the
curves $x_2 = mx_1^{\alpha/\beta}$.

Let $f(x)$ denote $1 + x^2$. An obvious candidate for $w$ is
$$w(x_1,x_2,x_3) = (x_1^{\alpha} + x_2^{\beta})f(x_3).$$
One checks that
\begin{align*}
 \det D^2w  &= |x_1|^{2\alpha-2}|x_2|^{\beta-2}\left(2\alpha\beta(\alpha-1)(\beta-1)f^2 - 4\alpha^2\beta(\beta-1)fx_3^2\right) \\
 &+ |x_1|^{\alpha-2}|x_2|^{2\beta-2}\left(2\alpha\beta(\alpha-1)(\beta-1)f^2 - 4\alpha\beta^2(\alpha-1)fx_3^2\right).
\end{align*}
Take $\alpha = 2-\delta$. Then for $|x_3|$ small depending on $\delta$ we have
$$\det D^2w \geq c(\delta)(|x_1|^{2\alpha-2}|x_2|^{\beta-2} + |x_1|^{\alpha-2}|x_2|^{2\beta-2}).$$
Along the curves $x_2 = mx_1^{\alpha/\beta}$, we compute
$$\det D^2w \geq c(\delta)(|m|^{\beta-2} + |m|^{2\beta -2}) \geq c(\delta),$$
since $1 < \beta < 2$. 

Thus, up to rescaling the $x_3$-axis and multiplying by a constant, we have
$$\det D^2w \geq 1 \quad \text{ in } \Omega = \{|x'|<1\} \times [-1,1].$$
 
\textbf{Construction of $S$:}
Let $\epsilon > 0$ be a small constant we will choose shortly depending on $\delta$.
Construct a self-similar set in $[-1/2,1/2]$ as follows: First, remove an open interval of length $\gamma = 1-2^{-3\epsilon}$ from the center. Proceed 
inductively by removing intervals a fraction $\gamma$ of each of those that remains. Denote the centers of the intervals removed at stage $k$ by 
$\{x_{i,k}\}_{i=1}^{2^{k-1}}$, and the intervals by $I_{i,k}$. Finally, let 
$$S = [-1/2,1/2] - \cup_{i,k} I_{i,k}.$$
It is easy to check that $|I_{i,k+1}| = \gamma2^{-(1+3\epsilon)k}$ and that $S$ has Hausdorff dimension $\frac{1}{1+ 3\epsilon}$.

\textbf{Construction of $v$:}
Let
$$ v_0(x) = \left\{
        \begin{array}{ll}
            |x| & \quad |x| \leq 1 \\
            2|x|-1 & \quad |x| > 1
        \end{array}
    \right.$$
We add rescalings of $v_0$ together to produce the desired function:
$$v(x) = \sum_{k=1}^{\infty}\sum_{i=1}^{2^{k-1}} 2^{-2(1+2\epsilon) k}v_0(2\gamma^{-1}2^{(1+3\epsilon)k}(x-x_{i,k})).$$
We now check that $v$ satisfies the desired properties:
\begin{enumerate}
 \item v is convex, as the sum of convex functions. Furthermore,
 \begin{align*}
  |v(x)| &\leq C\sum_{k=1}^{\infty}\sum_{i=1}^{2^{k-1}} 2^{-(1+\epsilon)k} \\
  &\leq C\sum_{k=1}^{\infty}2^{-\epsilon k},
 \end{align*}
 so $v$ is bounded.

 \item Let $x \in S$. We aim to show that $v$ separates from a tangent line more than
 $r^{2-\delta}$ a distance $r$ from $x$. By subtracting a line assume that $v(x) = 0$ and that $0$ is a subgradient at $x$.
 Assume further that $x+r < 1/2$ and that $2^{-(1+3\epsilon)k} < r \leq 2^{-(1+3\epsilon)(k-1)}$. 
 There are two cases to examine:
 
 \textbf{Case 1:} There is some $y \in (x+r/2,x+r) \cap S$. Then by the construction of $S$ it is easy to see that there is some interval $I_{i,k+2}$ such
 that $I_{i,k+2} \subset (x,x+r)$. On this interval, $v$ grows by
 $$2^{-2(1+2\epsilon)(k+2)} \geq cr^{2\frac{1+2\epsilon}{1+3\epsilon}} = cr^{2-\delta},$$
 where we choose $\epsilon$ so that
 $$\delta = \frac{2\epsilon}{1+3\epsilon}.$$

 \textbf{Case 2:} Otherwise, there is an interval $I_{i,j}$ of length exceeding $r/2$ such that $(x+r/2,x+r) \subset I_{i,j}$. In particular,
 $j \leq k+2$. Then at the left point
 of $I_{i,j}$, the slope of $v$ jumps by at least $2^{-(1+\epsilon)(k+2)}$. It follows that at $x+r$, $v$ is at least
 $$\frac{r}{2}2^{-(1+\epsilon)(k+2)} \geq cr^{2-\delta}.$$

 Thus, $v$ has the desired properties.
\end{enumerate}

\textbf{Construction of $u$:} We recall the following lemma on the solvability of the Monge-Amp\`{e}re equation (see \cite{Gut},\cite{Har}).
\begin{lem}\label{Solvability}
 If $\Omega$ is open, bounded and convex, $\mu$ is a finite Borel measure on $\Omega$ and $g$ is continuous and convex in $\overline{\Omega}$
 then there exists a unique convex solution $u \in C(\overline{\Omega})$ to the Dirichlet problem
 $$\det D^2u = \mu, \quad u|_{\partial\Omega} = g.$$
\end{lem}
Let $g(x_1,x_2,x_3) = C(v(x_1) + |x_2|)$ for a constant $C$ depending on $\delta$ we will choose shortly, 
and obtain $u$ by solving the Dirichlet problem
$$\det D^2 u = 1 \quad \text{ in } \Omega = \{|x'| < 1\} \times [-1,1], \quad \quad u|_{\partial \Omega} = g.$$
Take $z = (z_1,0,0)$ for $z_1 \in S$, and let $a_z$ be a subgradient of $v$ at $z_1$. Let
$$w_z(x) = g(z) + a_z(x_1-z_1) + w(x-z).$$
Since
$$w(x-z) \leq C_0(|x_1-z_1|^{2-\delta} + |x_2|^{\beta})$$
for some $C_0$, we can take $C$ large so that
$$g(x_1,x_2,\pm 1) \geq g(z) + a_z(x_1-z_1) + C(|x_1-z_1|^{2-\delta} + |x_2|) \geq w_z(x_1,x_2,\pm 1)$$
on the top and bottom of $\Omega$. Furthermore, since $g$ is independent of $x_3$ and for any fixed $x'$ we know $w_z$ takes its maxima at $(x',\pm 1)$,
we have $g \geq w_z$ on all of $\partial \Omega$.
Thus, $u \geq w_z$ in all of $\Omega$. Since $u$ takes the value $g(z)$ at $(z_1,0,\pm 1)$ and $w_z(z_1,0,x_3) = g(z)$ for all
$|x_3| < 1$, we have by convexity that $u = g(z)$ along $(z_1,0,x_3)$. 

We conclude that $\Sigma$ contains $S \times \{0\} \times (-1,1),$
which has Hausdorff dimension $1 + \frac{1}{1+3\epsilon} = 2-\frac{3}{2}\delta$. 

\begin{rem}
 To get the analagous example in $\mathbb{R}^n$, take
 $$u(x_1,x_2,x_3) + x_4^2 + ... + x_n^2.$$
 Observe that this solution has exactly the behavior described by Lemma \ref{NoQuadraticGrowth}, which says that $u$ must grow faster than
 quadratically in two directions. In the next section we show that for any $\epsilon$, these examples are not in $W^{2,1+\epsilon}$ for $\delta$ small
 enough.
\end{rem}

\section{$W^{2,1}$ Regularity}\label{W21Section}
In this section we obtain $W^{2,1}$ regularity for singular solutions to the Monge-Amp\`{e}re equation. Furthermore, by examining the examples in the previous section
we show that we cannot improve this result to $W^{2,1+\epsilon}$ regularity for an $\epsilon$ depending on $\lambda,\Lambda$ and $n$.

The following result of De Philippis, Figalli and Savin (see \cite{DFS}) gives $W^{2,1+\epsilon}$ regularity of solutions to $\lambda \leq \det D^2u \leq \Lambda$ in
compactly contained sections:
\begin{thm}\label{W21epsilon}
 Assume that 
 $$\lambda \leq \det D^2u \leq \Lambda \quad \text{ in } \Omega \text { and } S_h(x) \subset \subset \Omega.$$
 Then $u \in W^{2,1+\epsilon}(S_{h/2}(x))$ for some $\epsilon$ depending only on $\lambda,\Lambda$ and $n$.
\end{thm}

$W^{2,1}$ regularity then follows from our main theorem.

\begin{proof}[\textbf{Proof of Theorem \ref{W21}:}]
 We will show that $u \in W^{2,1}(B_{1/2})$. Local $W^{2,1}$ regularity follows from a standard covering argument.

 Theorem \ref{W21epsilon} gives local $W^{2,1}$ regularity on $B_1 - \Sigma$.
 By Theorem \ref{Main}, for any $\eta > 0$ we can cover
 $\Sigma \cap B_{1/2}$ by balls $\{B_{r_i}(x_i)\}$ with $r_i < 1/4$ such that
 $$\sum_{i=1}^{\infty} r_i^{n-1} < \eta.$$
 Let $A = \cup_{i=1}^{\infty} B_{r_i}(x_i)$. Since $u$ is a convex function, the second derivatives are controlled by $\Delta u$. It follows that
 \begin{align*}
  \int_{A} \|D^2u\| \, dx &\leq \int_{A} \Delta u \, dx \\
 &\leq \sum_{i=1}^{\infty} \int_{\partial B_{r_i}} u_\nu \,ds \\
 &\leq C\sum_{i=1}^{\infty} r_i^{n-1} \\
 &\leq C\eta,
 \end{align*}
 where $C$ is the Lipschitz constant of $u$ in $B_{3/4}$. This shows that the second derivatives cannot concentrate on $\Sigma$.
\end{proof}

 We now examine the integrability of $\Delta u$ for the examples constructed in the previous section. Fix a small $\delta$.
 We will show that for some $\epsilon$ small depending on $\delta$, we have $u \notin L^{1+\epsilon}$.
 (Note that this $\epsilon$ is not related to the one from the previous section). 

 On any ball $B_r$, by H\"{o}lder's inequality we have
 $$\int_{B_r} (\Delta u)^{1+\epsilon} \,dx \geq c(n) r^{-\epsilon n} \left(\int_{B_r} \Delta u \,dx\right)^{1+\epsilon}.$$
 Recall from the construction in the previous section that at points in $\Sigma$, $u$ grows
 from its tangent plane faster than $x_2^{\beta} = x_2^{1+\frac{\delta}{4-3\delta}}$ in the $x_2$ direction (at singular points, a translation and modification of $w$
 by a linear function touches $u$ by below).
 It follows that for $x \in \Sigma$ and $l_x$ a tangent plane to $u$ at $x$, we have
 $$\sup_{\partial B_r(x)}(u-l_x) \geq r^{\beta}.$$
 Applying convexity,
 \begin{align*}
  \int_{B_r(x)} (\Delta u)^{1+\epsilon} \,dx &\geq c(n)r^{-\epsilon n}\left(\int_{\partial B_r} u_{\nu} \,ds\right)^{1+\epsilon} \\
  &\geq c(n)r^{(n+\beta-2)(1+\epsilon)-\epsilon n} \\
  &\geq c(n)r^{n- 1 - \epsilon + (1+\epsilon)\frac{\delta}{3}}. 
 \end{align*}

 Fix $\eta$ small and cover $S \times \{0\} \times (-1,1)^{n-2}$ with balls of radius $r_i < \eta$. Take a Vitali subcover $\{B_{r_i}\}_{i=1}^{\infty}$. It follows that
 $$\int_{B_1} (\Delta u)^{1+\epsilon} \,dx \geq c(n) \sum_{i=1}^{\infty} r_i^{n-1-\epsilon+(1+\epsilon)\frac{\delta}{3}}.$$
 
 Taking $\epsilon = 4\delta$ above, we conclude that
 $$\int_{B_1} (\Delta u)^{1+\epsilon} \,dx \geq c(n)\sum_{i=1}^{\infty} r_i^{n-1-3\delta},$$
 where the expression on the right goes to $\infty$ as $\eta \rightarrow 0$ because the Hausdorff dimension of $S \times \{0\} \times (-1,1)^{n-2}$ is $n-1-\frac{3}{2}\delta$.
 Thus, $\Delta u$ is not $L^{1+\epsilon}$ for $\epsilon \geq 4\delta$.

 \begin{rem}
  In future work we intend to present a more precise version of Theorem \ref{Main} which gives $L\log L$ regularity of second derivatives of
 singular solutions to $$\lambda \leq \det D^2u \leq \Lambda.$$
 \end{rem}

\section{Unique Continuation}\label{UCsection}
Assume that $u,v$ satisfy the hypotheses of Theorem \ref{UniqueContinuation}.
For our proof of unique continuation we rely on the following classical unique continuation theorem for linear equations (see \cite{GL}):
\begin{thm}\label{ClassicalUC}
 Assume that $\Omega \subset \mathbb{R}^n$ is a connected open set and $u \in W^{1,2}_{loc}(\Omega)$ is a weak solution to the equation
 $$\partial_i(a^{ij}(x)u_j) + b^{i}(x)u_i + c(x)u = 0,$$
 where $a^{ij}(x)$ is Lipschitz and uniformly elliptic and $b^{i}(x), c(x)$ are bounded measurable. If $u = 0$ on some open subset of $\Omega$, then
 $u \equiv 0$ in $\Omega$.
\end{thm}
In \cite{AS}, the authors use the same theorem to prove unique continuation for fully nonlinear uniformly elliptic equations.
As in \cite{AS}, we note that Theorem \ref{ClassicalUC} also applies to classical solutions of nondivergence equations with Lipschitz
coefficients, which may be rewritten in the divergence form above. 
A more general version of this statement, proved using Carleman estimates, can be found in H\"{o}rmander's book \cite{H}, Theorem $17.2.6$.

We will apply this result to the difference of $u$ and $v$, which solves a linear equation where $u$ and $v$ are sufficiently regular. Indeed, suppose
$u$ and $v$ are $C^2$ in a neighborhood of $x$ and let $w_t$ be the convex combination $tu + (1-t)v$. Let $(W_t)^{ij}$ be the matrix of cofactors 
for $D^2w_t$. Then by expanding $0 = \int_{0}^1 \frac{d}{dt} \det D^2w_t dt$ we get
$$a^{ij}(x)(u-v)_{ij} = 0,$$
where
$$a^{ij}(x) = \int_{0}^1 (W_t)^{ij}(x)dt.$$

The regularity theory of Caffarelli \cite{C2} allows us to use this observation at points of strict convexity for solutions to the Monge-Amp\`{e}re equation:
\begin{thm}\label{W2p}
Assume
$$\det D^2u = f \quad \text{ in } \Omega, \quad \quad u|_{\partial\Omega} = 0$$
where $f \in C^{1,\alpha}(\Omega)$ is strictly positive. Then 
$$u \in C^{3,\alpha}(\Omega).$$
\end{thm}

Finally, we observe that open sets whose complements have zero Hausdorff $n-1$ dimensional measure are connected.

\begin{lem}\label{Connected}
Assume $K \subset \mathbb{R}^n$ is closed, and assume further that $\mathcal{H}^{n-1}(K) = 0$. Then $\mathbb{R}^n - K$ is pathwise connected.
\end{lem}
\begin{proof}
 Assume by way of contradiction that $D = \mathbb{R}^n-K$ is not pathwise connected. Since $D$ is open, by rotation, translation and scaling we can assume that the points
 $\pm Re_n \in D$ cannot be connected by any continous path through $D$ and that
 $$\{|x'| < 1\} \times \{\pm R\} \subset D.$$

 Let $K'$ be the projection of $K$ onto $\{x_n = 0\}$ and let $B_1' = B_1 \cap \{x_n=0\}$. If $B_1'-K' \neq \emptyset$, 
 this would violate the contradiction hypothesis because then we could find a point $x' \in B_1'$ such that $(x',t) \in D$ for all $t \in \mathbb{R}$ and 
 take our path to be the straight lines from $-Re_n$ to $(x',-R)$ to $(x',R)$ to $Re_n$.

 We conclude that for any cover of $K$ by balls $\{B_{r_i}(x_i)\}_{i=1}^{\infty}$, we have
 $$\sum_{i} r_i^{n-1} \geq 1,$$
 contradicting that $\mathcal{H}^{n-1}(K) = 0$. 
\end{proof}

The proof of unique continuation follows easily from these observations and our main theorem.

\begin{proof}[\textbf{Proof of Theorem \ref{UniqueContinuation}:}]
Let $\Sigma_u$ and $\Sigma_v$ be the singular sets of $u$ and $v$ respectively, and let $A = \Omega - (\Sigma_u \cup \Sigma_v)$.
Since $A$ is dense in $\Omega$, it suffices to show that $u = v$ on $A$. 

By Caffarelli's theory (\cite{C1}), $A$ is an open set. Indeed, for $x \in A$ we can find some $p$ in $\mathbb{R}^n$ and $h$ small such that $S_{h,p}^u(x) \subset\subset\Omega$, 
and since $f$ is bounded in this section Caffarelli gives that $u$ is strictly convex in a neighborhood of $x$. (The same reasoning gives that $v$ is strictly convex in a neighborhood of $x$.)
By Theorem \ref{Main}, the complement of $A$ has Hausdorff $n-1$ dimensional measure zero, so by Lemma \ref{Connected}, $A$ is connected.
By Theorem \ref{W2p}, the difference $u-v$ satisfies the linear equation
$$a^{ij}(x)(u-v)_{ij} = 0$$
on $A$, where $a^{ij}$ are locally uniformly elliptic and $C^{1,\alpha}$ in $A$. The conclusion follows from Theorem \ref{ClassicalUC}.
\end{proof}


 \section*{Acknowledgement}
 I would like to thank my thesis advisor Ovidiu Savin for his patient guidance and for his feedback on the the drafts of this paper. 
 I would also like to thank Jun Kitagawa, Nam Le and Yu Wang for helpful comments.

 The author was partially supported by the NSF Graduate Research Fellowship Program under grant number DGE 1144155. 


\frenchspacing
\bibliographystyle{plain}

\end{document}